\documentclass[runningheads,a4]{llncs}
\usepackage{graphicx,amssymb,amsmath}
\usepackage[a4paper,hmargin=1.3in,vmargin=1.2in]{geometry}
\usepackage[colorlinks=false,bookmarks=false,hyperfootnotes=false]{hyperref}
\def\etal{et~al.}

\graphicspath{{figures/}}

\begin{document}
\mainmatter

\title{Flip Graphs of Bounded Degree Triangulations}

\author{Oswin Aichholzer\inst{1} \and Thomas Hackl\inst{1} \and David Orden\inst{2} \and Pedro Ramos\inst{2} \and \\G\"unter Rote\inst{3} \and Andr\'e Schulz\inst{4} \and Bettina Speckmann\inst{5}}

\institute{Institute for Software Technology, Graz University of Technology, Austria\\
  \email{[oaich|thackl]@ist.tugraz.at}
\and Departamento de Matem\'aticas, Universidad de Alcal\'a, Spain\\
  \email{[david.orden|pedro.ramos]@uah.es}
\and Institute of Computer Science, FU Berlin, Germany \\
  \email{rote@inf.fu-berlin.de}
  \and Institut f\"ur Mathematische Logik und Grundlagenforschung, Universit\"at M\"unster, Germany \\
  \email{andre.schulz@uni-muenster.de}~{\sf Tel:	+49 251 83-32686  Fax:	+49 251 83-33078}
\and Department of Mathematics and Computer Science, TU Eindhoven, the Netherlands\\
  \email{speckman@win.tue.nl}
}

\authorrunning{O. Aichholzer, T. Hackl, D. Orden, P. Ramos, G. Rote, A. Schulz, B. Speckmann}

 
\maketitle

\begin{abstract}\noindent
  We study flip graphs of triangulations whose maximum vertex degree
  is bounded by a constant $k$. In particular, we consider
  triangulations of sets of $n$ points in convex position in the plane
  and prove that their flip graph is connected if and only if $k > 6$;
  the diameter of the flip graph is $O(n^2)$. We also show that, for
  general point sets,
  flip graphs of pointed pseudo-triangulations can be disconnected
  for $k \leq 9$, and
  flip graphs of triangulations can be disconnected for any~$k$.

  Additionally, we consider a relaxed version of the original problem. We allow the violation of the degree bound $k$ by a small constant. Any two triangulations with maximum degree at most $k$ of a convex point set are connected in the flip graph by a path of length $O(n \log n)$, where every intermediate triangulation has maximum degree at most $k+4$.
\end{abstract}


\section{Introduction}\label{sec:introduction}

An edge flip is a common local and constant size operation that
transforms one triangulation into another. If two adjacent triangles form a convex
quadrilateral, a flip removes their common edge and replaces it with the other diagonal of the
convex quadrilateral.
The \emph{flip graph}  of triangulations
of a planar point set $S$ has a vertex for every triangulation of $S$,
and two vertices are connected by an edge if there is a flip that
transforms the corresponding triangulations into each other. One of
the first and most fundamental results concerning edge flips in
triangulations is the fact that flips can be used repeatedly to
convert any triangulation into the Delaunay
triangulation~\cite{Sipson1973}. This implies
that the flip graph of planar triangulations is connected for any planar point set~$S$.

The \emph{flip distance} between two triangulations is the minimum
number of flips needed to convert one triangulation into the
other. The diameter of the flip graph, which is an upper bound on the flip
distance, is known to be $\Theta(n)$ if $S$ is in convex
position~\cite{STT88}, and $\Theta(n^2)$ if $S$ is in general position (see for example the book of Edelsbrunner~\cite[page 11]{E06}).
However, computing the flip distance between two particular triangulations
seems to be difficult, and only partial results are known:
Hanke \etal~\cite{hos-efdt-96} proved an upper bound of the flip distance in terms
of the number of crossings between the edges of the triangulations, while
Eppstein~\cite{e-hefg-07} provides an $O(n^2)$ algorithm which computes the flip
distance between two triangulations for a class of very special point sets: those
having no empty convex pentagons (collinear points are allowed).
%
%
In higher dimensions the flip graph does not  have to be
connected~\cite{Santos2000}.

Of growing interest are subgraphs of flip graphs which correspond
to particular classes of triangulations. Houle \etal~\cite{hhnr-gtpf-05} consider triangulations which contain a
perfect matching of the underlying point set. They show that this
class of triangulations is connected via flips, that is, the
corresponding subgraph of the flip graph is connected. Related results
exist for order-$k$ Delaunay graphs, which consist of a subset of
$k$-edges, where a $k$-edge is an edge for which a covering disk
exists which covers at most $k$ other points of the set. For general
point sets the graph of order-$k$ Delaunay graphs is connected via
edge flips for $k \leq 1$, but there exist examples for $k \geq 2$
that can not be converted into each other without leaving this
class~\cite{abghnr-dga-08}. If the underlying point set is in convex
position,  the resulting flip
graph is connected for any $k \geq 0$~\cite{abghnr-dga-08}. The flip operation has also been
extended to other planar graphs, see Bose and Hurtado~\cite{bh-fpg-09} for a
recent and extensive survey.

%
%
%

There are point sets for which every triangulation has a vertex of
degree $n-1$,
%
%
but point sets in convex position always have triangulations with maximum
vertex degree~4. Therefore, we concentrate on point sets in convex position and study the following
question: is the flip graph of triangulations with maximum vertex
degree $k$ connected? Triangulations of sets of~$n$ points in convex
position are in one to one correspondence to  many other combinatorially equivalent structures, all of them interpreting Catalan numbers~\cite[Chapter~6]{stanley}.
Our results can be reinterpreted via these equivalences. Maybe the most prominent of these structures
are binary search trees over a set of $n$ elements stored in the leaves. A degree-$k$ bounded triangulation corresponds to a search tree, whose elements are connected to its neighbors by a path of at most $k$ vertices. A flip in a triangulation translates to the standard rotation in a search tree.

\smallskip\noindent
{\bf Results.}
Let $S$ be a set of $n$ points in convex position in the plane. In
Section~\ref{sec:convex} we show  that
the flip graph of triangulations of $S$ with maximum vertex degree at most $k$ is disconnected if $k\leq 6$.
Then we prove that the flip graph is connected for any $k > 6$ and
its  diameter is $O(n^2)$.
In Section~\ref{sec:relaxed} we improve the diameter
bound to $O(n\log n)$  allowing the violation of the degree bound by $4$ in
intermediate steps.
%
Finally, in Section~\ref{sec:general}, we consider point sets
in general position and show that flip graphs of the considered triangulations can be
disconnected for any~$k$. Moreover, we show that the flip graph of
so-called pointed pseudo-triangulations
can be disconnected for $k\leq 9$ when the maximum vertex degree is required to be at most~$k$.

Partial results of this work have been presented at the \emph{European Conference on Combinatorics, Graph Theory and Applications (EuroComb)}, in Bordeaux, France, in September 2009.
An extended abstract without proofs was published in the corresponding conference proceedings~\cite{ahorrss-fgbdt-09}.
Preliminary results for flipping while maintaining a relaxed degree bound
(Section~\ref{sec:relaxed}) appeared before in the PhD thesis of one of the authors~\cite{h-rlt-10}.
%

\section{Flipping with tight degree bounds}
\label{sec:convex}

Let $\mathcal{T}_k(S)$ be the set of triangulations of
$S$, such that all points have degree at most~$k$. We denote with ${\cal F}_k(S)$ the flip graph over the ground set  $\mathcal{T}_k(S)$.
Two vertices are connected in ${\cal F}_k(S)$ if their triangulations differ by a flip.
%
%
If clear from the context, the reference to $S$ will be omitted. As mentioned above, arbitrary point sets do not necessarily have a
triangulation of bounded vertex degree, but point sets in convex
position always have triangulations of maximum vertex degree~4. An example is depicted in
\figurename~\ref{fig:lower}(a).

%

\smallskip\noindent {\bfseries Definitions and notation.}
Throughout this section and Section~\ref{sec:relaxed}, $S$ is a
set of $n$ points in convex position in the plane.
%
%
Let $D$ be the dual graph of a triangulation $T$ of $S$. Since $S$
is in convex position, $D$ is a tree.
We distinguish three different types of triangles in $T$: \emph{ears},
which have two edges on the convex hull of $S$, \emph{path triangles},
which have one edge on the convex hull of $S$, and \emph{inner
  triangles}, which have no edge on the convex hull of $S$.  The
\emph{tip} of an ear is the vertex that is incident to two convex
hull edges.  The ears of $T$ are dual to the \emph{leaves} of $D$
and inner triangles of $T$ are dual to \emph{branching vertices} of $D$ of
degree three. A \emph{path} in $D$ is any connected subgraph of~$D$
that consists only of vertices of degree two or one, thus
every vertex of a path is dual to a path triangle or an ear.
A path that contains a leaf is called a \emph{leaf
  path}. An inner
triangle whose corresponding branching vertex in $D$ is adjacent to at least two leaf paths is named
\emph{merge triangle}.
A set of triangles in $T$ that is dual to a path in~$D$ is called a \emph{strip}. If the strip is dual to a leaf path
 that is adjacent to a branching vertex we call the strip an \emph{ear strip}.

%

We define two special types of strips: fans and zigzags.
A {\em fan} is a maximal subset of at least two triangles that
 all share one common vertex of
$S$, the so-called \emph{fan handle}, and whose convex hull edges are consecutive.
%
%
The size of a fan is the number of
triangles it consists of.
A strip is said to form a \emph{zigzag} if the
deletion of the convex hull edges leaves a path. Notice that an ear can also be a zigzag, and a zigzag and a fan might share a triangle.
If the zigzag is not a part of a larger zigzag it is called \emph{maximal}.
%
The flipping of every other diagonal of a zigzag is called an
\emph{inversion} of the zigzag.

A triangulation whose triangles form a single zigzag is called a \emph{zigzag triangulation}.
Such a triangulation has exactly two ears.
A zigzag
triangulation is uniquely defined (up to an inversion) by the location of
one of its ears. We select one of the zigzag triangulations as the \emph{canonical triangulation}.
A triangulation is called a \emph{fringe triangulation}, if
every ear strip is a zigzag. In particular, every zigzag
triangulation is a fringe triangulation.

\begin{theorem}\label{thm:456}
Let $S$ be a set of $n$ points in convex position. The flip graph ${\cal F}_k(S)$ is disconnected for $k\in\{4,5,6\}$.
\end{theorem}
\begin{proof}
Every convex point set with $n$ points has  $\Theta(n)$ different zigzag triangulations. Clearly,  one cannot flip even a single
edge in such a triangulation without exceeding a vertex degree of 4. See \figurename~\ref{fig:lower}(a).
\begin{figure}[htb]
  \centering
  \includegraphics[width=.8\columnwidth]{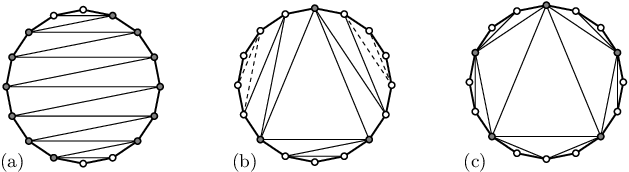}
  \caption{Triangulations with maximum vertex degree $k = 4$ (a), $k = 5$ (b), and $k = 6$ (c).}
  \label{fig:lower}
\end{figure}
For maximum degree $k = 5$, consider the triangulation shown in
\figurename~\ref{fig:lower}(b). Only the dashed edges can be flipped, but
there are $\Theta(n)$ rotationally symmetric versions of this
triangulation, none of which can be reached from any other without
exceeding a vertex degree of 5. For $k = 6$ consider the triangulation
depicted in \figurename~\ref{fig:lower}(c). No edge of this triangulation can
be flipped without violating the degree bound but again there are $\Theta(n)$ rotationally symmetric
versions of this triangulation, none of which can be reached from any
other without exceeding a vertex degree of 6.
\end{proof}

In contrast to the negative result of Theorem~\ref{thm:456} we prove
as a positive result the following theorem in the remainder of this section.
\begin{theorem}\label{theorem:main}
  Let $S$ be a set of $n$ points in convex position. For any
  $k>6$ the flip graph ${\cal F}_k(S)$ is connected and has diameter  $O(n^2)$.
\end{theorem}

\smallskip
\noindent{\bfseries Proof Outline.}
We prove Theorem~\ref{theorem:main} by showing how to flip a triangulation $T\in\mathcal{T}_k(S)$ $(k>6)$ to
the canonical triangulation with no more than $O(n^2)$ flips while maintaining the degree bound.
As a first step we show in
Subsection~\ref{subsec:fringe} how to flip $T$ to a fringe
triangulation. In Subsection~\ref{subsec:zigzag} we
prove that we can construct  a \emph{light} merge
triangle, that is, a merge triangle with at least two vertices of
degree smaller than~$k$. We then show how to remove this light merge triangle by
merging its adjacent ear strips (zigzags) with $O(n)$ flips, resulting in a
triangulation that has one less inner triangle. This triangulation can
again be converted into a fringe triangulation without introducing new inner triangles. After repeating this
step $O(n)$ times we have converted $T$ into a zigzag
triangulation.  Finally, in
Subsection~\ref{subsec:rotate} we demonstrate how to ``rotate'' any
zigzag triangulation to the canonical triangulation of $S$
using $O(n)$ flips.

\subsection{Creating a fringe triangulation}\label{subsec:fringe}

\begin{lemma}\label{lem:fans}
  Let $S$ be a set of $n$ points in convex position and let
  $T\in\mathcal{T}_k(S)$, $k>6$. Then $T$ can be transformed into a fringe
  triangulation of $S$ in $O(n^2)$ flips, while at no time exceeding a
  vertex degree of $k$.
%
\end{lemma}
\begin{proof}
\begin{figure}[htb]
    \centering
    \includegraphics[width=.65\columnwidth]{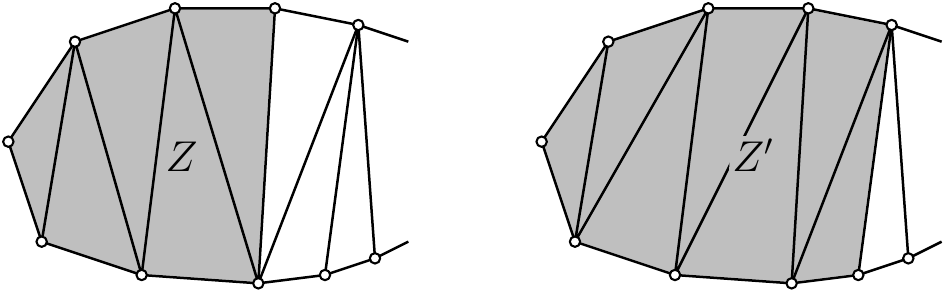}
    \caption{A zigzag can be ``extended'' by an inversion.}
    \label{fig:fan}
\end{figure}
Recall that in a fringe triangulation every ear strip is a
zigzag. If $T$ is not a fringe triangulation, then it has at least one ear strip $E$
that is not a zigzag. Let $Z$ be the maximal zigzag that  contains the ear of  $E$. The end triangle of $Z$ that is not the ear is called the \emph{terminal triangle}. The terminal triangle is also part of a fan, otherwise $Z$ would not have been a maximal zigzag or $E$ would be a zigzag. Let $v_h$ be the fan handle of this fan.

We perform the inversion of $Z$ that removes the diagonal that contributes to the degree of $v_h$.
This does not violate the degree bound, since only vertices with degree at most three increase their degree.
By inverting the zigzag we changed the situation on the terminal triangle. Before the inversion this triangle was also part of a fan. The inversion of $Z$, however, is a zigzag that can be extended by at least one more path triangle. This induces a new maximal zigzag $Z'$ inside $E$ that contains its ear. The number of triangles of $Z'$ exceeds the number of triangles of $Z$. We can repeat this procedure until we get a zigzag, whose terminal triangle is next to a merge triangle. By this we have transformed $E$ into a zigzag (see Figure~\ref{fig:fan} for an illustration). We repeat this for every ear strip that is not a zigzag.

An inversion includes at most $O(n)$ flips. Since we extend the zigzags that contain the ears of $T$ every time we invert, we process at most $O(n)$ inversions, which leads to a bound of $O(n^2)$ flips for flipping to a fringe triangulation.
\end{proof}

In order to flip to the canonical triangulation, we first flip to a fringe triangulation, but not all intermediate triangulations
in the flip sequence will be fringe triangulations. In particular, we might have to flip to a fringe triangulation every time
we removed a merge triangle. The next lemma allows us to analyze the accumulated costs of these flips. We define as an \emph{outer triangle}
 a triangle that is contained in a zigzag, which contains an ear.
 
\begin{lemma}\label{lem:fansacc}
Let $T$ be a triangulation with $\ell$ outer triangles. Then $T$ can be flipped to a fringe triangulation $T_f$ with no more than $n(\ell'-\ell)$ flips,
where $\ell'>\ell$ is the number of outer triangles in $T_f$.
\end{lemma}
\begin{proof}
We flip from $T$ to $T_f$ by converting every ear strip into a zigzag as explained in the proof of Lemma~\ref{lem:fans}. 
Assume that some ear is incident to a maximal zigzag of length $m$ and the corresponding ear strip contains $m'>m$ triangles. Every inversion we perform increases the number 
of outer triangles. Thus, $m'-m$ inversions suffice 
to turn this ear strip into a zigzag. The statement of the lemma follows, since 
an inversion needs at most $O(n)$ flips.
\end{proof}

%
%

\subsection{Merging zigzags}\label{subsec:zigzag}

Recall that a merge triangle is an inner triangle that is adjacent to
at least two ear strips. In a fringe triangulation, all ear strips are zigzags.
We call a merge triangle \emph{light}, if two of its vertices have degree less than $k$.
We first show that every fringe triangulation can be flipped to a fringe triangulation with a light merge triangle without violating the degree bound.
 Then we show how to remove a light
merge triangle by merging two of its adjacent ear strips.

\begin{lemma}\label{lem:mergeexists}
  Let $S$ be a set of $n$ points in convex position and let
  $T\in\mathcal{T}_k(S)$ be a fringe triangulation of $S$ with maximum
  vertex degree $k > 6$ that is not a zigzag triangulation. Then $T$ can be flipped in $O(n)$ flips to another fringe triangulation that has a light merge triangle.
\end{lemma}
\begin{proof}
  Since $T$ is a fringe but not a zigzag triangulation, it has at least one merge
  triangle. Let $D$ be the dual graph of $T$, and let $D'$ be the tree that is obtained by deleting all leaf paths from $D$.
  We assume that $D'$ has at least three nodes. Otherwise, $T$ would have only one merge triangle, or exactly two adjacent merge triangles. In both cases it is easy to create a triangulation with no vertex of degree larger than~$6$ by inverting the incident zigzags appropriately. We define as $D''$ the tree obtained by deleting all leaves from $D'$. Let $\Delta''$ be a leaf in $D''$ and let $\Delta'$ be a leaf in $D'$ adjacent to $\Delta''$ (see \figurename~\ref{fig:dualgraphs}). By construction $\Delta'$ is incident to two zigzags, with common point $v_{\text{tip}}$.
%
%
  It is easy to see that $v_{\mathrm{tip}}$ has degree at most~$6$. If
  the degree of $v_{\mathrm{tip}}$ is smaller than~$6$ (and a zigzag
  contains more triangles than just the ear) we invert the incident
  zigzag to increase the degree (by this the degree of one of the
  other vertices of $\Delta'$ decreases).
  Since all vertices of a zigzag, except the vertices of $\Delta'$,
  have degree at most $4$, this inversion maintains the degree bound.
  \begin{figure}[htb]
    \centering
    \includegraphics{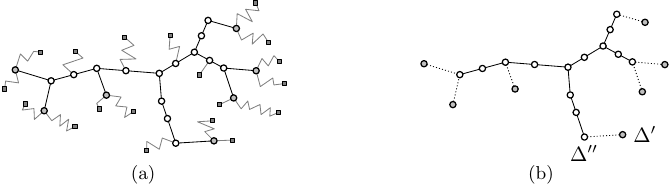}
    \caption{(a) In graph $D$ internal paths are shown as single
      edges; (b) Graphs $D'$ and $D''$ (without gray nodes and dotted
      edges).}
    \label{fig:dualgraphs}
  \end{figure}
  It remains to show that one of the other vertices of $\Delta'$ has degree at most $6$. To prove this we distinguish three cases.\\
  {\bf Case 1:  $\Delta''$ is a path triangle.} As a consequence the vertex of $\Delta'$ that is incident to the convex hull edge of $\Delta''$ has degree $4$ and therefore $\Delta'$ is a light  merge triangle.\\
  {\bf Case 2: $\Delta''$ is a merge triangle and the degree of its
    dual in $D'$ is $2$.} In this case (see
  \figurename~\ref{fig:merge-triangle}, left) $\Delta''$ is adjacent to an
  ear strip. Let $v$ be the vertex of the corresponding zigzag that is
  also a vertex of $\Delta'$. The degree of $v$ is at most $6$ and hence $\Delta'$ is a light merge triangle.\\
  {\bf Case 3: $\Delta''$ is a merge triangle and the degree of its
    dual in $D'$ is $3$.} This situation is depicted in
  \figurename~\ref{fig:merge-triangle}, right. The triangle $\Delta''$ has
  to be adjacent to a merge triangle other than $\Delta'$ that is also
  a leaf in $D'$. We call this triangle $\Delta_s$. Let $v$ be the
  common vertex of $\Delta',\Delta''$, and $\Delta_s$, and let $u$ be
  the common vertex of $\Delta_s$ and its two adjacent ear strips. The
  degree of $v$ could be~$7$ but in this case we can invert the zigzag
  incident to the edge $uv$ and reduce the degree of $v$ to~$6$. The
  inversion increases by~$1$ the degree of $u$, which was at most~$5$
  before. Hence, $\Delta'$ is a light merge triangle.
\begin{figure}[htb]
  \centering
  \includegraphics[width=.7\columnwidth]{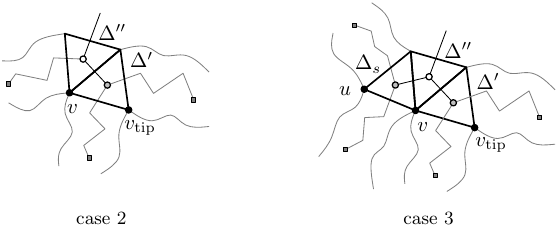}
  \caption{The triangle $\Delta''$ (dual to a leaf of $D''$) can
    have 1 (case 2) or 2 (case 3) children in $D'$.}
  \label{fig:merge-triangle}
\end{figure}

In any case we perform only a constant number of inversions and thus
only a total of $O(n)$ flips.
\end{proof}
%
%

By Lemma~\ref{lem:mergeexists} we can flip every triangulation to a fringe triangulation with
a light merge triangle $\Delta$. We now show how to remove
$\Delta$ by merging its adjacent zigzags $Z_1$ and $Z_2$. Throughout the merging we
denote with $e_1$ the edge that separates $Z_1$ and $\Delta$, and with $e_2$ the edge that separates $Z_2$ and $\Delta$.
Further, let us denote by~$e$ the third edge of~$\Delta$, and by $v_1$ and $v_2$ the endpoints of~$e$,
such that before the merging $v_1$ is part of $Z_1$ and $v_2$ is part of $Z_2$ (see \figurename~\ref{fig:mergezz}(a)). 
The vertex on $\Delta$ that is not on $e$ is named $v_t$.
As $\Delta$ is light we can assume that
$v_1$ has degree at most~$k$, and that $v_2$ has degree at most~$k-1$.
\begin{lemma}\label{lem:merging}
  Let $S$ be a set of $n$ points in convex position,
  let~$T\in\mathcal{T}_k(S)$ be a fringe triangulation of $S$,
  let~$\Delta$ be a light merge triangle of $T$, and let
  $\widetilde{n}$ be the total number of vertices of $\Delta$ and
  its two adjacent ear strips. These ear strips and
  $\Delta$ can be merged into one zigzag ending in an ear with
  $O(\widetilde{n})$ flips, while at no time exceeding a vertex degree
  of $k>6$.
%
\end{lemma}
\begin{proof}
%
We show how to merge $Z_1$ and $Z_2$ together with $\Delta$ into a new
zigzag $Z'$ that starts at $e$. We ensure that the degree of $v_t$ is $6$ by inverting
$Z_1,Z_2$ if necessary.
As a first, preprocessing, step we flip $e_1$ followed by $e_2$.
These flips create the first two triangles of $Z'$ and do not
violate the degree bound. Now there is
a new triangle that lies between (the remains of) $Z_1$
and $Z_2$, with new edges~$e_1$ and~$e_2$. We finish the preprocessing phase by flipping (the new edge) $e_1$ and then (the new edge) $e_2$. The result is depicted in \figurename~\ref{fig:mergezz}(b).
The two vertices $v_1$ and $v_2$ which may have high
vertex degree have been ``cut off''. Notice that if $Z_1$ and $Z_2$ are small the light merge triangle might already be gone.

\begin{figure}[htb]
  \centering
\begin{tabular}{cccc}
\includegraphics[page=1,width=.22\columnwidth]{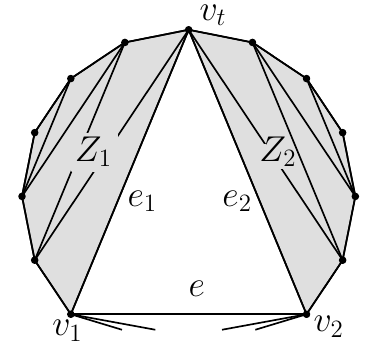} &
\includegraphics[page=2,width=.22\columnwidth]{light} &
\includegraphics[page=3,width=.22\columnwidth]{light} &
\includegraphics[page=4,width=.22\columnwidth]{light} \\
(a) & (b) & (c) & (d)
\end{tabular}
\caption{Removing a light merge triangle. The start configuration (a), after the preprocessing step (b), after flipping $F_3$ (c), and after flipping $F_3$ a second time (d).}
\label{fig:mergezz}
\end{figure}%

Let us pause for a second and see what we have obtained so far. Both, $Z_1$ and $Z_2$ have lost two triangles each. On top we obtained a new ear. This is the new zigzag $Z_t$ that will grow during the flipping sequence. Thus we have two shrinking zigzags $Z_1$ and $Z_2$, and two growing zigzags $Z'$ and $Z_t$. All four zigzags are separated by two adjacent triangles, whose common diagonal is called $d$. We will maintain this configuration throughout the flipping. Notice that the edge labels might change after a flip.

We continue by flipping $e_1$, followed by $d$, and then followed by
$e_2$. We call this flip sequence $F_3$. After flipping $F_3$, the
zigzags $Z_1$ and $Z_2$ have been shrunk by one triangle each. On the
other hand, $Z'$ has been enlarged by two triangles (see
\figurename~\ref{fig:mergezz}(c)). Flipping again $F_3$ will further
decrease $Z_1$ and $Z_2$ but this time we add two triangles to the
zigzag $Z_t$ (see \figurename~\ref{fig:mergezz}(d)). We repeat flipping
$F_3$ until all the triangles of at least one of the zigzags $Z_1$ and
$Z_2$ are consumed. If both zigzags vanish at the same time, we
successfully removed the light merge triangle. If only one zigzag (say $Z_2$)
vanishes, then we created another light merge triangle $\Delta'$. 
We would like to remove this
triangle again by first flipping the preprocessing sequence
followed by iterations of $F_3$. However, the situation at $\Delta'$ might differ
from the start configuration at $\Delta$. Let $v_t'$ be the vertex 
of $\Delta'$ that is incident to $Z_t$ and $Z_1$. For its counterpart  $v_t$ we assumed  that 
its degree  is exactly 6. This is not necessarily true 
for $v'_t$. In order to make $v'_t$ a degree 6 vertex, we could invert $Z_1$ and $Z_t$
appropriately. Inverting the zigzag $Z_t$ can be afforded, since the size of this zigzag is proportional  
to the number of flips we just did (every other time we flip $F_3$, we increase $Z_t$ by two triangles). 
Thus we can assume that the degree of $v'_t$ is at least 5.
Consider now the situation where $Z_t$ contributes to the vertex degree of $v'_t$, but 
the degree of $v_t'$ is 5 (see Fig.~\ref{fig:merging-invert}(a)).
Notice that it might be too costly to invert $Z_1$. Therefore, instead of inverting $Z_1$, we flip the diagonal of $\Delta'$ that is also 
part of $Z_1$ (Fig.~\ref{fig:merging-invert}(b)) and then invert $Z_t$ including the newly added triangle (Fig.~\ref{fig:merging-invert}(c)).
Again the costs of inverting $Z_t$ can be charged to the flips of $F_3$. By this we construct a new light merge triangle $\Delta''$, with the desired degrees.
We consecutively remove the light merge triangles, until the merging is complete.
\begin{figure}[htb]
  \centering
\begin{tabular}{ccc}
\includegraphics[page=1,width=.3\columnwidth]{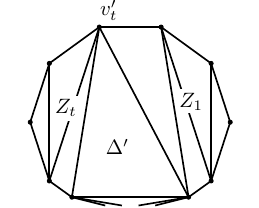} &
\includegraphics[page=2,width=.3\columnwidth]{../merging-invert} &
\includegraphics[page=3,width=.3\columnwidth]{../merging-invert} \\
(a) & (b) & (c) 
\end{tabular}
\caption{The flipping sequence that constructs a new light merge triangle with the right degree bounds.}
\label{fig:merging-invert}
\end{figure}%

Every other time we flip $F_3$ (including the possible flips we charged for the inversions of $Z_t$), we
add two triangles to $Z'$. Thus we need only a constant number of flips to enlarge $Z'$, which shows 
that after  $O(\widetilde{n})$ flips we
removed the light merge triangle.
\end{proof}

\subsection{Rotating a zigzag triangulation}\label{subsec:rotate}

Recall that a zigzag triangulation is uniquely defined (up to an
inversion) by the location of one of its ears. A {\em rotation} is a
sequence of flips that transforms one zigzag triangulation into another
one. We introduce  a
rotation operation that needs only a linear number of flips.

\begin{figure}[htb]
  \centering
  \begin{tabular}{ccc}
  \includegraphics[width=.32\columnwidth,page=1]{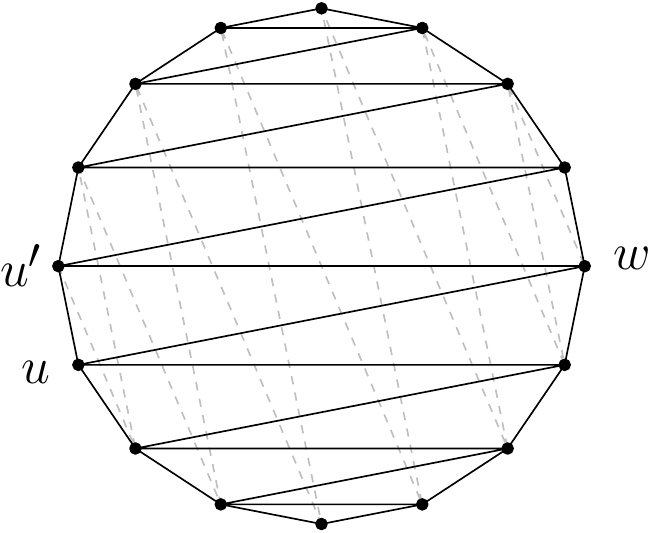} & &\hskip6ex
  \includegraphics[width=.32\columnwidth,page=2]{impr-rot}
  \\ (a) & &(b)
  \end{tabular}
  \caption{Rotating a zigzag using O(n) flips as described in Lemma~\ref{lem:rotate}. The dashed edges in (a) depict the target zigzag triangulation $T'$.}
  \label{fig:linrot}
\end{figure}

\begin{lemma}\label{lem:rotate}
  Let $S$ be a set of $n$ points in convex position and let $T$ be any
  zigzag triangulation of $S$. $T$ can be flipped into any other
  zigzag triangulation of $S$ with $O(n)$ flips, while at no time
  exceeding a vertex degree of $k>6$.
\end{lemma}

\begin{proof}
  Let $u$ be the tip of an ear of the target zigzag
  triangulation. Further, let $u'$ be a vertex next to $u$ on the
  convex hull of $T$. Denote with $\Delta$ the triangle in $T$,
  which is incident to the edge $(u,u')$. \figurename~\ref{fig:linrot}(a)
  shows the situation, where the dashed edges present the target
  zigzag triangulation and $\Delta$ in $T$ is spanned by $u$, $u'$,
  and $w$.

The removal of the triangle $\Delta$ would split the zigzag
triangulation $T$ into the two zigzags $Z_1$ and $Z_2$. As a first
step we invert $Z_1$ and $Z_2$. The result is shown in
\figurename~\ref{fig:linrot}(b). In this configuration $\Delta$ is a
degenerate light merge triangle. (Think of a zigzag of size zero glued
to $(u,u')$.) Therefore, we can apply Lemma~\ref{lem:merging} and merge $Z_1$ and $Z_2$ into one
single zigzag $Z$. The zigzag $Z$ induces a zigzag triangulation that
has an ear with tip $u$. Either this triangulation coincides with
$T'$, or we obtain $T'$ after a single inversion of $Z$.

The flipping sequence uses $O(n)$ flips to invert $Z_1$ and
$Z_2$. Applying Lemma~\ref{lem:merging} needs again $O(n)$ flips, as
well as the possible final inversion of $Z$, and thus
 we need $O(n)$ flips in total.
\end{proof}

%
%


%

We have now obtained all lemmas to give a proof for
Theorem~\ref{theorem:main}.
\begin{proof}[of Theorem~\ref{theorem:main}]
  Lemma~\ref{lem:fans} states that any triangulation with maximum
  vertex degree $k>6$ can be transformed into a fringe triangulation
  with $O(n^2)$ flips. A fringe triangulation  contains only
  a linear number of inner triangles. By Lemma~\ref{lem:mergeexists} at least one of them can be turned into a light
  merge triangle with $O(n)$ flips. This light merge
  triangle can be ``resolved'' by merging its two adjacent ear strips as explained in~Lemma~\ref{lem:merging}.
    By constructing and merging light merge triangles in succession we remove all inner triangles.
  This
  takes $O(n)$ flips per merge triangle, according to
  Lemma~\ref{lem:merging}, and thus sums up to a total of $O(n^2)$ flips.
  Notice that after the merge the resulting triangulation might contain a single ear strip that is not a zigzag. 
  However, after the merge the number of outer triangles was not decreased.
  By repeated inversions we can turn this triangulation into a fringe triangulation with the technique presented in Lemma~\ref{lem:fans}.
  Let $k_1,k_2,\ldots$ be the number of outer triangles for each fringe triangulation that appears
  in the flip sequence (in order). By Lemma~\ref{lem:fansacc} the total costs of flipping to a fringe
  triangulations are at most $\sum_i n(k_{i+1}- k_i)$.
  The accumulated costs for recreating a fringe triangulation after a merge are therefore  $O(n^2)$.
  Lemma~\ref{lem:rotate} concludes the proof by showing that
  each zigzag triangulation can be rotated to the canonical zigzag
  triangulation using $O(n)$ flips. In addition each operation used
  (set of flips) respects the degree bound, as stated in the
  corresponding lemmas.
\end{proof}

\section{Point sets in convex position with relaxed degree bounds}
\label{sec:relaxed}

If we allow a slightly increased degree bound in intermediate steps
(when flipping from one degree bounded triangulation to another) then
we can show an upper bound  of $O(n\log n)$ for the diameter of the flip graph. To
this end, we will show that in this relaxed setting we can remove a
constant fraction of the inner triangles per step, involving a linear
number of flips.  For this we need to slightly redefine fringe
triangulations, and improve the handling of fans and the rotation of
zigzag triangulations.

\subsection{Relaxed fan handling}\label{subsec:fringenew}

We first introduce a new class of triangulations that will play the role of the fringe triangulation in the previous section.
\begin{definition}\rm
We call a triangulation a \emph{relaxed fringe triangulation} if, and only if,
\begin{enumerate}
\item[(a)] it is a fringe triangulation, i.e., every ear strip is a zigzag,
\item[(b)] it contains no fan of size greater than~$3$,
\item[(c)] every fan is adjacent to an inner triangle.
\end{enumerate}
\end{definition}

Let us now show how to update Lemma~\ref{lem:fans}. To achieve a shorter flip sequence in the lemma we allow to flip to a relaxed fringe triangulation.
\begin{lemma}\label{lem:fans2}
  Let $S$ be a set of $n$ points in convex position and let
  $T\in\mathcal{T}_k(S)$. Then $T$ can be transformed into a relaxed fringe
  triangulation of $S$ in $O(n)$ flips, while at no time exceeding a
  vertex degree of $k>6$.
%
\end{lemma}
\begin{figure}
  \centering
  \includegraphics{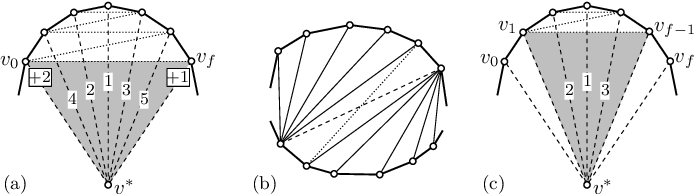}
  \caption{Handling fans: The boxed numbers show the degree change,
    the other numbers indicate the order in which flips are
    performed.}
  \label{fig:fan2}
\end{figure}
\begin{proof}
The flipping sequence consists of two parts. In a preprocessing step we flip all edges that are contained in two fans with different fan handles.
Such a pair of fans is depicted in \figurename~\ref{fig:fan2}(b). After flipping the common edge, the adjacent fans will either vanish, or be separated by a zigzag.

We continue by removing all fans of size greater than~$3$.
Consider a fan, $F$, of size $f$ and with handle $v^*$, like the one
depicted in \figurename~\ref{fig:fan2}(a). The fan $F$ contains a chain of the convex hull of $S$.
Let~$v_0,v_f$ be the endpoints of that chain.

Assume first that neither $v_0$ nor $v_f$ are vertices of an inner triangle. Since $v_0$ and $v_f$ cannot be  fan handles
their degree is at most $4$. Therefore we can convert $F$ to an
inner triangle $\Delta(v^*v_{f}v_0)$ and a zigzag containing an ear, performing $O(f)$ flips (see
\figurename~\ref{fig:fan2}(a)).

Assume now that either  one of the vertices $v_0$ or $v_f$  is incident to an inner triangle.
If $F$ has size smaller than $4$, we leave the fan unaltered. Otherwise, we treat the strip spanned by $v^*, v_1,\dotsc v_{f-1}$ as a (sub)fan of size $f-2$.
Since the vertices $v_1$ and $v_{f-1}$ have degree three, this ``subfan'' can be removed by the procedure
in the previous paragraph.
In particular  we can flip $F$ to an inner
triangle $\Delta(v^*v_{f-1}v_1)$ and a zigzag adjacent to the edge $(v_1,v_{f-1})$ as depicted in \figurename~\ref{fig:fan2}(c).

We end with a triangulation, whose ear strips contain a single fan of size three or smaller. By the technique described in Lemma~\ref{lem:fans}, such a fan can be removed by a series of inversions. Thus, an ear strip of size $n'$ can be turned into a zigzag by $O(n')$ flips. As a consequence, all ear strips can be transformed into zigzags using $O(n)$ flips. Furthermore we observe that
the preprocessing sequence flips at most $n$ edges and the removal of a fan of size $f$ takes $O(f)$ flips. Thus a relaxed fringe triangulation can be reached in $O(n)$ flips.
\end{proof}

\subsection{Merge two inner triangles to one}

Consider the dual graph $D$ of a relaxed fringe triangulation $T$. The
reduced graph $D'$ is obtained by removing all leaf paths from
$D$. Each leaf in $D'$ is a merge triangle. In the relaxed setting all
merge triangles are light. Inner triangles, that are dual to nodes of
degree $2$ in $D'$, are incident to one ear strip in $T$. Remember
that all ear strips are zigzags. We call two of these inner triangles
{\em consecutive} if the triangulation between them is a strip. We
show how to merge two consecutive inner triangles into one, thereby
reducing the total number of inner triangles in $T$ by one.

  \begin{figure}[htb]
    \centering
    \includegraphics[width=.45\columnwidth]{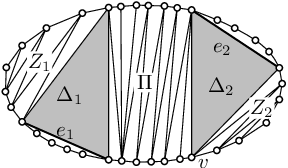}
    \caption{Example of two consecutive  inner triangles, $\Delta_1$
      and $\Delta_2$, with degree $2$ in the reduced dual graph $D'$.}
    \label{fig:merge-relaxed}
  \end{figure}
Let $\Delta_1$ and $\Delta_2$ be two consecutive inner triangles
that have degree $2$ in $D'$.
Let $\Pi$ be the strip between~$\Delta_1$ and~$\Delta_2$, and let
$Z_1$ ($Z_2$) be the maximal zigzag, which is an ear strip, incident to
$\Delta_1$ ($\Delta_2$).
Further let $e_1$, $e_2$ be the edges of $\Delta_1$ and
$\Delta_2$, respectively, that are not part of a triangle of $\Pi$,
$Z_1$, or $Z_2$. See \figurename~\ref{fig:merge-relaxed} for an example.
The diagonals $e_1$ and $e_2$ delimit a connected subset
$\widetilde{T}$ of $T$, which is the union of $\Delta_1$,
$\Delta_2$, and the triangles of $Z_1$, $Z_2$, and $\Pi$. We denote
the size of $\widetilde{T}$ with $\widetilde{n}$.


\begin{lemma}\label{lem:lincombine}
  Let $\widetilde{T}$ be a subset of a relaxed fringe triangulation $T$ as
  defined in the above paragraph. We can transform~$T$ into a relaxed fringe
  triangulation, such that the two inner triangles $\Delta_1$ and
  $\Delta_2$ get merged to at most one inner triangle. This
  reduces the number of inner triangles by at least $1$ and takes
  $O(\widetilde{n})$ flips, while at no time exceeding a vertex degree
  of $k+4$.
\end{lemma}

\begin{proof}
  We only consider the subset $\widetilde{T}$ which can be seen as
  cutting $T$ along $e_1$ and $e_2$, turning these diagonals into convex
  hull edges of $\widetilde{T}$.
  Observe that $\widetilde{T}$ is dual to a path.  As $T$ is a relaxed fringe
  triangulation, $\widetilde{T}$  contains at most two fans of constant size.
  Thus, $\widetilde{T}$ can be converted into a zigzag triangulation
  with a constant number of inversions, and a total of $O(\widetilde{n})$
  flips.
  This zigzag triangulation can be rotated, such that it starts at
  $e_1$. By Lemma~\ref{lem:rotate} only $O(\widetilde{n})$ flips
  are needed for this rotation.

  This transforms $T$ into a relaxed fringe triangulation $T'$. Between $e_1$
  and $e_2$, $T'$ has a zigzag starting at $e_1$ and ending at an
  inner triangle, $\Delta$, at $e_2$, and another zigzag dual to a
  leaf path incident to $\Delta$. Observe that, if the chains of
  the convex hull of $\widetilde{T}$ between $e_1$ and $e_2$ are of
  equal size (up to $\pm 1$), then the resulting triangulation between
  $e_1$ and $e_2$ is a zigzag (without inner triangles).

  We now argue that we have respected the degree bound during the
  rotation and inversion of a zigzag
  (inside $\widetilde{T}$). The only vertices that can
  accumulate a degree bigger than~$k$ are the vertices of $e_1$ and
  $e_2$.
  These vertices have a degree of at least $3$ in $\widetilde{T}$
  before the transformation. $\widetilde{T}$ contains only fans of
  constant size (less than $4$) and the operations used (inversion,
  rotation) generate a degree of at most $7$ in the present case (see
  for example the left-most vertex in
  \figurename~\ref{fig:mergezz}(c)).
  Thus, during the transformation, the vertices of $e_1$ and $e_2$ get
  an overall degree in~$T$ of at most $k+4$. The result of the merge
  is again a relaxed fringe triangulation with vertex degree at most $k$, as
  the result of transforming $\widetilde{T}$ is a zigzag
  triangulation. (If one of the vertices of $e_1$ and $e_2$ had a
  degree of $3$ in $\widetilde{T}$, it may be necessary to invert one (or both) of the
  zigzags in $T'$ between $e_1$ and $e_2$ to restore the vertex degree
  bound. This  takes only another $O(\widetilde{n})$ flips.) 
\end{proof}

\subsection{Putting things together}

We now have all tools to show an upper bound of $O(n\log
n)$ for the flip distance of bounded degree triangulations with
relaxed intermediate degree bounds.

\begin{theorem}\label{thm:nlogn}
  Let $S$ be a set of $n$ points in convex position and let
  $T\in\mathcal{T}_k(S)$ be a triangulation of $S$ with maximum vertex
  degree $k > 6$. Then $T$ can be flipped into the canonical
  triangulation of $S$ in $O(n\log n)$ flips while at no time
  exceeding a vertex degree of $k+4$.
\end{theorem}

\begin{proof}
  By Lemma~\ref{lem:fans2} we can transform $T$ into a relaxed fringe
  triangulation $T'$ as defined in Section~\ref{subsec:fringenew}. We
  show that we always can remove at least $1/8$  of the inner
  triangles of $T'$ with a linear number of flips. The result is again  a relaxed
  fringe triangulation.

  Consider the tree $\hat D$ whose nodes are the inner triangles of
  $T'$. Two nodes in $\hat D$ are connected by an edge if they
  correspond to consecutive inner triangles.
  Let $d_1$, $d_2$, $d_3$ be the number of degree $1$, $2$, $3$ nodes
  in $\hat D$, and let $m=d_1+d_2+d_3$ be the number of inner
  triangles in $T'$.  Assume first that $\hat D$ has at least $m/8$
  leaves. Every leaf of $\hat D$ corresponds to a light merge triangle for
  the relaxed degree bound. By Lemma~\ref{lem:merging} we can
  therefore merge $m/8$ light merge triangles with $O(n)$ flips. After
  the merging we reduced the number of inner triangles by a factor of
  $m/8$. Notice that after the merges we might have to apply inversions to
  reconstruct a relaxed fringe triangulation. These costs are counted
  separately. The number of outer triangles is non-decreasing throughout 
  the flipping sequence, therefore, due to Lemma~\ref{lem:fansacc},
   the accumulated costs for the extra inversions require $O(n^2)$ flips.
   
  Assume now that there are less than $m/8$ leaves in $\hat D$.
  Observe that in a tree with no nodes of degree two, the number of
  leaves exceeds the number of inner nodes, and therefore in our case
  we have $d_3 < d_1 < m/8$.  Since $\hat D$ is a tree it has $m-1$
  edges. The number of edges with a vertex of degree one or three is
  less than $d_1+3 d_3 < m/2$. Hence, we have more than $m/2$ edges
  where both endpoints have degree two. Let this edge set be $E_2$,
  with $|E_2| > m/2$. The tree $\hat D$ decomposes into three
  edge-disjoint matchings $M_1, M_2, M_3$. Such a decomposition can be
  determined greedily. Assume that $M_1$ is the matching that contains
  the majority of edges from $E_2$. For every edge in $M_1\cap E_2$ we
  can apply Lemma~\ref{lem:lincombine} independently. The combined
  flipping sequence has $O(n)$ flips in total. The number of inner
  triangles is reduced by a factor of $|M_1\cap E_2| \geq |E_2|/3 >
  m/6$.

  Thus, in both cases, we can reduce the number of inner triangles by
  at least a factor of $m/8$ per $O(n)$ flips. This gives an overall
  bound of $O(n\log n)$ flips.
\end{proof}

\section{Point sets in general position}\label{sec:general}

\begin{figure}[htb]
  \centering
  \includegraphics{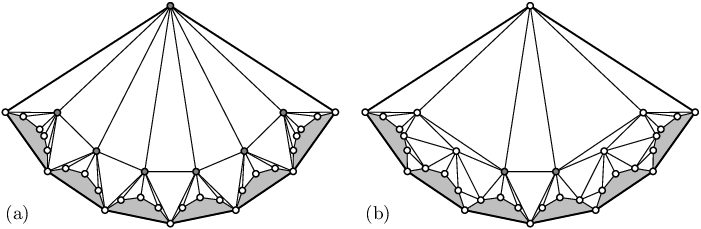}
  \caption{Two triangulations which cannot be flipped into each other.}
  \label{fig:triexample1}
  \vspace{-\baselineskip}
\end{figure}
In this section we study flip graphs of bounded degree triangulations
of a set $S$ of $n$ points in general position in the plane.
As in Section~\ref{sec:convex}, let $\mathcal{T}_k(S)$ be the set of
triangulations of $S$ such that all vertices have degree at most $k$.
We study the same question as in the convex setting:
If there are two triangulations $T_1\in\mathcal{T}_k(S)$ and
$T_2\in\mathcal{T}_k(S)$, is it possible to flip from $T_1$ to $T_2$
while at no time exceeding a vertex degree of~$k$?
%
%

We can answer this question negatively for any
$k$.
For $k<7$ we know from Section~\ref{sec:convex} that the flip graph of
$\mathcal{T}_k(S)$ may be disconnected. In the following we show that
for any $k\geq 7$ there exists a point set which has two
triangulations $T_1\in\mathcal{T}_k(S)$ and $T_2\in\mathcal{T}_k(S)$
which cannot be flipped into each other without exceeding a vertex
degree of $k$. Consider the example for $k = 8$ depicted in
\figurename~\ref{fig:triexample1}. The shaded parts represent zigzag
triangulations and the dark vertices have degree~8. In the left
triangulation only edges of the zigzags can be flipped without
exceeding vertex degree 8. Hence it is impossible to reach the
triangulation on the right. It is not difficult to obtain similar
examples for any $k \geq 7$.

Finally, we consider the flip graph connectivity problem for a relaxation of
triangulations.
\emph{Pseudo-triangulations} (see~\cite{Rote2008} for a survey)
generalize triangulations in the following sense: A
\emph{pseudo-triangle} is a planar polygon with exactly three internal
angles less than $180^\circ$. A pseudo-triangulation of a point set
$S$ is a partition of the convex hull of $S$ into pseudo-triangles
whose vertex set is $S$. A pseudo-triangulation is called
\emph{pointed}, if every vertex is incident to an angle greater than
$180^\circ$.  Two pseudo-triangles which share an edge form a
(possibly degenerate) pseudo-quadrilateral $Q$, that is a polygon with
4 internal angles less than $180^\circ$.  We call the vertices that
realize the small angles the \emph{corners} of $Q$.  The common edge
$e$ of the pseudo-triangles is the only non-polygon edge of a geodesic
between two opposing corners in $Q$. The geodesic between the other
two corners consists also of only one non-polygon edge, say
$e'$. Exchanging $e$ by $e'$ is called a flip in a pointed
pseudo-triangulation (see \figurename~\ref{fig:ptflip}).  The induced flip
graph ${\cal F}_{PT}(S)$ of pointed pseudo-triangulations of a point
set $S$ is connected and its diameter is $O(n \log
n)$~\cite{Bereg2004}.
\begin{figure}[h]
  \centering
  \includegraphics[width=.6\columnwidth]{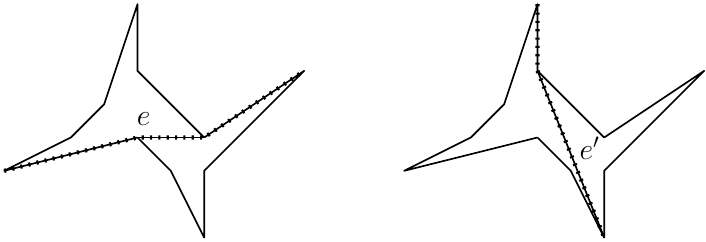}
  \caption{A flip in a pointed pseudo-triangulation. The dotted lines denote the geodesics between opposing corners of the pseudo-quadrilateral.} \label{fig:ptflip}
\end{figure}
%
%
Further, it is known that any point set $S$ in
general position has a pointed pseudo-triangulation of maximum vertex
degree 5~\cite{kkmsst-tbptp-03}. Hence, the question arises if there is a $k \geq 5$ such that
the flip graph of pointed pseudo-triangulations with maximum vertex
degree $k$ is connected.

\begin{figure}[h]
  \centering
  \includegraphics{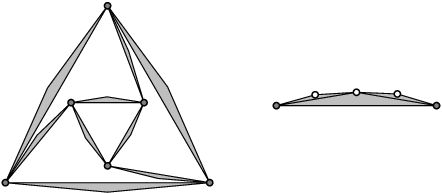}
  \caption{The ``triangular edges'' in the left drawing consist of the
    structure shown on the right, with the indicated orientation. Dark
    vertices have degree 9.} \label{fig:degree9}
\end{figure}
For point sets in convex position every pointed pseudo-triangulation
is in fact a triangulation. Thus our results on triangulations of
convex point sets apply also for pointed pseudo-triangulations.
For point sets in general position the question whether the flip graph of pointed pseudo-triangulations with maximal vertex degree  $k$ is connected remains open. However, we can show, that for $k\leq 9$ the induced flip graph is disconnected. Consider the pointed pseudo-triangulation
$P$ depicted in \figurename~\ref{fig:degree9}. $P$ has maximal vertex degree
9, but no edge of $P$ can be flipped without creating a vertex with degree 10.
However, $P$ is  not the only pointed pseudo-triangulation of this point set with maximum
vertex degree 9. Examples for $k=7,8$ can be easily derived from this construction.

\section{Conclusions and future work}\label{sec:conclusion}

We considered the flip graph of degree bounded (pseudo-)triangulations
with respect to the maximum vertex degree.
For point sets $S$ in convex position we showed that for the set
$\mathcal{T}_k(S)$ of triangulations with maximum vertex degree $k$,
the flip graph is not connected for $k\leq6$, but connected for any
$k>6$.
That is, for any $T_1,T_2\in\mathcal{T}_k(S)$, $k>6$, there exists a
sequence of flips from $T_1$ to $T_2$ such that each intermediate
triangulation~$T_j$ is also in $\mathcal{T}_k(S)$.
We were able to bound the length of such a
sequence by $O(n^2)$.


It is an open question whether the bound on the flip distance can be further
improved. Theorem~\ref{thm:nlogn} gives hope that a bound of $O(n\log n)$ on the diameter of  $\mathcal{T}_k(S)$ is within reach.
 The $k+4$ relaxation for intermediate triangulations stems
mainly from the intention of keeping the proofs as simple as
possible. We are certain that a closer investigation will lead to a
reduction of the additive term, maybe even to zero.
Further, we were able to prove that the flip graph is not connected
for triangulations $T\in\mathcal{T}_k(S)$ of point sets $S$ in general
position, for any (constant) $k$. There is a slight chance that the flip graph remains
connected when allowing a relaxed intermediate degree bound.


For degree bounded pointed pseudo-triangulations in general point sets
we showed, that for a maximum vertex degree of at most $9$ the flip
graph is not connected. Hence, for degree bounded pointed
pseudo-triangulations it is only known that, if the flip graph is
connected, then the maximum vertex degree has to be at least 10.
Further, for pointed pseudo-triangulations another interesting
question arises, when bounding the face-degree, that is the maximum
number of vertices of pseudo-triangles. Every point set has a pointed
pseudo-triangulation with face degree at most
$4$~\cite{kkmsst-tbptp-03}. For any $k\ge4$ it is open if ${\cal
  F}_{PT}$ stays connected when deleting all vertices that correspond
to pointed pseudo-triangulations which contain a face of degree
greater than~$k$.

\section{Acknowledgments}\label{sec:ack}

This work was initiated during the Second European
Pseudo-Triangulation Working Week in Alcal{\'{a}} de Henares, Spain,
2005. We want to thank the organizers and all participants for a
stimulating environment and fruitful discussions. Moreover, we thank 
Francisco Santos for helpful comments on this subject.

O.~Aichholzer was funded by the Austrian Science Fund (FWF):
I 648-N18 and S09205-N12.
T.~Hackl was funded by the Austrian Science Fund (FWF): P23629-N18.
D.~Orden and P.~Ramos were partially supported by MICINN grants MTM2008-04699-C03-02/MTM and MTM2011-22792 and by the ESF EUROCORES programme EuroGIGA, CRP ComPoSe, under grant
EUI-EURC-2011-4306.
G.~Rote and A.~Schulz were partially supported by the Deutsche
Forschungsgemeinschaft (DFG) under grant \mbox{RO 2338/2-1}.
B.~Speckmann was partially supported by the Netherlands' Organisation
for Scientific Research (NWO) under project no.~639.022.707.

\bibliographystyle{abbrv}      
\bibliography{FlipGraphs}   

\end{document}